\documentclass[11pt]{article}

\usepackage{amsmath}
\usepackage{amssymb}

\usepackage{amsthm}
\theoremstyle{definition}
\newtheorem{theorem}{Theorem}[section]
\newtheorem{definition}[theorem]{Definition}
\newtheorem{proposition}[theorem]{Proposition}
\newtheorem{corollary}[theorem]{Corollary}
\newtheorem{remark}[theorem]{Remark}
\newtheorem{lemma}[theorem]{Lemma}
\newtheorem{assumption}[theorem]{Assumption}

\usepackage[ruled, linesnumbered]{algorithm2e}
\SetKwInput{KwSample}{Sample}

\usepackage[T1]{fontenc}
\usepackage{lmodern}

\usepackage{tikz}
\usepackage{pgfplots}
\pgfplotsset{compat=1.18}
\usetikzlibrary{arrows.meta}
\usetikzlibrary{backgrounds}
\usepgfplotslibrary{patchplots}
\usepgfplotslibrary{fillbetween}
\pgfplotsset{%
    layers/standard/.define layer set={%
        background,axis background,axis grid,axis ticks,axis lines,axis tick labels,pre main,main,axis descriptions,axis foreground%
    }{
        grid style={/pgfplots/on layer=axis grid},%
        tick style={/pgfplots/on layer=axis ticks},%
        axis line style={/pgfplots/on layer=axis lines},%
        label style={/pgfplots/on layer=axis descriptions},%
        legend style={/pgfplots/on layer=axis descriptions},%
        title style={/pgfplots/on layer=axis descriptions},%
        colorbar style={/pgfplots/on layer=axis descriptions},%
        ticklabel style={/pgfplots/on layer=axis tick labels},%
        axis background@ style={/pgfplots/on layer=axis background},%
        3d box foreground style={/pgfplots/on layer=axis foreground},%
    },
}
\usepackage{adjustbox}
\usepackage{subcaption}

\providecommand{\sep}{, }
\providecommand{\keywords}[1]{\textbf{\textit{Keywords:}} #1}

\newcommand{\setto}{\rightrightarrows}

\newcommand{\lparan}{\symbol{40}}
\newcommand{\rparan}{\symbol{41}}

\newcommand{\bbR}{\mathbb{R}}
\newcommand{\bbB}{\mathbb{B}}
\newcommand{\bbN}{\mathbb{N}}
\newcommand{\bbC}{\mathbb{C}}
\newcommand{\bbE}{\mathbb{E}}
\newcommand{\bbP}{\mathbb{P}}
\newcommand{\bbH}{\mathbb{H}}
\newcommand{\bbRn}{\mathbb{R}^n}
\newcommand{\bbRnxn}{\mathbb{R}^{n {\times} n}}

\newcommand{\xbar}{\bar{x}}

\newcommand*{\diff}{\mathop{}\!\mathrm{d}}

\newcommand{\aseq}{\quad \text{(a.s.)}}

\DeclareMathOperator{\dist}{dist}
\DeclareMathOperator{\GL}{Gl}
\DeclareMathOperator{\tr}{tr}
\DeclareMathOperator{\SO}{SO}
\DeclareMathOperator{\Unif}{Unif}
\DeclareMathOperator{\Id}{I}
\DeclareMathOperator{\proj}{proj}

\DeclareMathOperator*{\argmin}{argmin}
\DeclareMathOperator*{\minim}{minimize}
\newcommand{\minprob}[2]{\begin{array}{ll}%
  \displaystyle{\minim_{#1}}\quad& #2 \\%
  \end{array}}

\usepackage{hyperref}
\usepackage{url}

\author{Titus Pinta}
\title{A Stochastic Newton-type Method for Non-smooth Optimization}

\begin{document}
\maketitle
\begin{abstract}
  We introduce a new framework for analyzing \lparan{}Quasi-\rparan{}Newton type methods
  applied to non-smooth optimization problems. The source of randomness comes
  from the evaluation of the (approximation) of the Hessian. We derive, using
  a variant of Chernoff bounds for stopping times,
  expectation and probability bounds for the random variable representing the number
  of iterations of the algorithm until approximate first order optimality
  conditions are validated.
  As an important distinction to previous results in the literature, we do
  not require that the estimator is unbiased or that it has finite variance.
  We then showcase our theoretical results in a stochastic Quasi-Newton method
  for X-ray free electron laser orbital tomography and in a sketched Newton
  method for image denoising.
\end{abstract}

\keywords{Newton's Method\sep%
  Stochastic Optimization\sep%
  Higher Order Methods\sep%
  Nonsmooth Optimization\sep%
  Subspace Methods}

\section{Introduction}
Stochastic optimization is a critical approach to numerous problems of current
interests. The randomness can come from sampling large data sets, as in
machine learning, or from inherent properties of the physical systems
modeled, as in quantum mechanical simulations. While for deterministic
and unconstrained problems, (semismooth) Newton's method reigns supreme in
terms of speed
of convergence and quasi-Newton methods shine due to the performance per
cost ratio, it seems that the stochastic optimization community has focused
primarily on gradient descent based approaches. The aim of this work is
to bridge this gap and to bring the tools of non-smooth \lparan{}Quasi-\rparan{}Newton-type
methods to the world of stochastic optimization.
Numerical experiments in tomography and image denoising confirm the validity and
applicability of the proposed approach.

Frameworks for the analysis of stochastic \lparan{}Quasi-\rparan{}Newton-type methods applied
to non-smooth and non-convex problems have been
studied and developed in~\cite{CaoBerSch24Firs,
  HaiLuoZhi21App,YanMilWenZha22Asto,MilXiaWenUlb22Onth,
  MilXIaCenWenUlb19Asto,BulPatKshShaSreNatWoo21Asto,NaDerMah23Hess}.

In our work, we are interested in a general optimization problem,
\begin{equation*}
  \min_{x \in U \subseteq \bbRn} f(x),
\end{equation*}
with a (nonsmooth) objective, $f:\bbRn \to \bbR$.
The difficulty stems from the fact that
higher order information such as Hessian, directional derivatives, or second
order Clarke
subdifferential, is only accessible via a possibly biased stochastic oracle.
Such examples
abound in the sciences, where the objective can be constructed statistically
from large datasets or where the randomness is intrinsic in physical systems.
Another source of randomness comes from stochastic approximations
of the higher order information as encountered in sketching methods, based
on the Johnson-Lindenstrauss Embedding Lemma.

For this purpose, we propose an algorithm based on a backtracking strategy
for accepting a step.
The final section of this work presents some concrete practical applications
of our algorithm.

In order to fully analyze such frameworks, the machinery of stochastic
processes has to be involved. In the rest of this section, we recall some
of the key aspects of stochastic processes required in the rest of this work.
The second section deals with the notion of regularity that is behind our
analysis and derives some properties of the Newton-like update, under this
regularity assumption. The third section assembles everything together into
our main result.

\subsection{Stochastics}
We assume that the reader is familiar with the standard notions of
probability theory, such as random variables, expected values, conditional
expectations and so on. For a standard reference covering these topics,
the reader should consult~\cite{Dur90Prob}.

In this work, we consider the standard Euclidean
spaces equipped with the standard Borel sigma algebra, $\sigma$ and the standard
cylindrical sigma algebra defined on the space of functions between Euclidean
spaces (so a random function is a measurable function, with respect to the
cylindrical sigma algebra). For the definitions
in this section, $\pi$ is an arbitrary probability measure.
To set the stage and to fix some notation, we provide the following definitions.
\begin{definition}
  A function $X:\bbN \to (\bbRn, \sigma, \pi)$, usually denoted as a sequence ${\{X^k\}}_{k \in \bbN}$,
  is called a {\em stochastic process}. A random
  variable, $K \in \bbN$, is called a {\em stopping time\/} for the process $X$ if for
  any $k \in \bbN$, the random variable $I\{K > k\}$ is independent from any $X^{n}$
  with $n \ge k$.
\end{definition}
The following properties of a special class of stopping times are well known.
\begin{proposition}[Hitting Times are Stopping Times]\label{prop:hitting times}
  Let ${\{T^k\}}_{k \in \bbN} \in {\lbrack}0, \infty{\rparan}$ be a sequence of independent random variables.
  Let $\varepsilon > 0$ and assume that $\bbE(T^k) < \infty$ and $\bbE(T^k) > \varepsilon$ for all $k$.
  Consider the stochastic process ${\{S^k\}}_{k \in \bbN}$ defined by
  \begin{equation*}
    S^n = \sum_{k = 0}^{n - 1}T^k.
  \end{equation*}
  Then $K$, defined by
  \begin{equation*}
    K = \min \{n \in \bbN~|~S^n \ge \alpha\},
  \end{equation*}
  is a stopping time for $S^k$ and for any $\alpha \ge 0$. Furthermore $\bbE(K) < \infty$.
  Such a stopping time is called a {\em hitting time}.
\end{proposition}

The first statement we recall is a classic result from martingale theory, namely
Doob's optional stopping theorem (see~\cite{GriSti01Prob} for the proof).
\begin{theorem}[Doob's Optional Stopping Theorem]
  Let ${\{X^k\}}_{k \in \bbN} \ge 0$ be a stochastic process with $\bbE(X^k) < \infty$ for all $k$ and
  let $K$ be a stopping time for this process with $\bbE(K) < \infty$. Assume that for all
  $k$
  \begin{equation*}
    \bbE(X^k|X^{k-1},\dots,X^{0}) \le X^{k - 1}.
  \end{equation*}
  Then $X^K$ is bounded almost surely and $\bbE(X^K) \le \bbE(X^0)$.
\end{theorem}

We are also going to use the well known behavior of the expectation of
a particular kind of stochastic process evaluated at a stopping time.
\begin{proposition}[Expectation of Sums]\label{prop:expectation of sums}
  Let ${\{T^k\}}_{k \in \bbN} \ge 0$ be a sequence of independent random variables,
  with $\bbE(T^k) \ge \alpha$ for all $k$. Let $S^{n} = \sum_{k = 0}^{n - 1}T^k$ be a
  stochastic process and $K$ a stopping time for ${\{S^k\}}_{k \in \bbN}$ with
  $\bbE(K) < \infty$.
  Then
  \begin{equation*}
    \bbE(S^K) \ge \bbE(K)\alpha.
  \end{equation*}
\end{proposition}

The next result is due to Young~\cite{You21Cher}. We adapted the proof
of a much more general statement to fit better into our context.
\begin{proposition}[Chernoff Bound for Stopping Times]%
  \label{prop: chernoff bound for stopping times}
  Let ${\{T^k\}}_{k \in \bbN}$ be a sequence of independent Bernoulli random variables,
  with $\bbP(T^k = 1) \ge 1 - \delta \ge \frac{1}{2}$ for all $k$.
  Let $S^{n} = \sum_{k = 0}^{n - 1}T^k$ be a stochastic process and $K$ a stopping time
  for ${\{S^n\}}_{n \in \bbN}$ with $E < \infty$. Then for any $\gamma \in (0, 1)$ and for any $E \in \bbN$,
  \begin{equation*}
    \bbP((1 - \gamma)K \ge 2 S^K + \gamma(1 - \gamma^2)E) \le e^{-\gamma^2 E}.
  \end{equation*}
\end{proposition}
\begin{proof}
  We define the stochastic process ${\{\Phi^k\}}_{k \in \bbN}$ by $\Phi^0 = e^{-\gamma^2 E}$ and
  \begin{equation*}
    \Phi^N = {(1 + \gamma)}^{\sum_{k = 0}^{N - 1} (1 - T^k)}
    {(1 - \gamma)}^{\sum_{k = 0}^{N - 1} (T^k)}e^{-\gamma^2 E}.
  \end{equation*}
  Let $N \in \bbN$. Using the inequality
  \begin{equation*}
    \forall x, y, \varepsilon \in [0, 1],\quad {(1 \pm \varepsilon)}^x \le 1 \pm \varepsilon x,
  \end{equation*}
  we can bound, recognizing the fact that $\forall k \in \bbN$, $T^{k} \in [0, 1]$,
  \begin{align}\label{eq:chernoff:bound on fraction}
    \frac{\Phi^N}{\Phi^{N - 1}}
    &= {(1 + \gamma)}^{1 - T^{N-1}}{(1 - \gamma)}^{T^{N - 1}}
      \le(1 + (1 - T^{N-1})\gamma)(1 - \gamma T^{N-1}) \nonumber \\
    &\le1 + ((1 - T^{N-1}) - T^{N-1})\gamma - \gamma^2T^{N-1}(1 - T^{N-1}) \nonumber \\
    &\le1 + ((1 - T^{N-1}) - T^{N-1})\gamma.
  \end{align}
  From the assumption, for any $k \in \bbN$,
  \begin{equation*}
    \bbE(T^{k}) \ge \frac{1}{2} \ge \bbE(1 - T^k),
  \end{equation*}
  and equivalently
  \begin{equation*}
    \bbE(1 - T^k - T^k) \le 0.
  \end{equation*}
  Because $T^{N-1}$ is independent on all $T^k$ with $k < N - 1$,
  \begin{equation*}
    \bbE(1 - T^{N-1} - T^{N-1}|T^{N-2},\dots,T^0) = \bbE(1 - T^{N-1} - T^{N-1}) \le 0.
  \end{equation*}
  We can take conditional expectations in~\eqref{eq:chernoff:bound on fraction}
  \begin{equation*}
    \bbE\left(\frac{\Phi^N}{\Phi^{N - 1}}|T^{N-2},\dots,T^0\right)
    \le 1 + \gamma\bbE((1 - T^{N-1}) - T^{N-1}|T^{N-2},\dots,T^0) \le 1.
  \end{equation*}
  Because $\Phi$ is defined in terms of $T$ and $\Phi^0$ is a constant, we can
  rewrite the previous equation as
  \begin{equation*}
    \bbE\left(\frac{\Phi^N}{\Phi^{N - 1}}|\Phi^{N-1},\dots,\Phi^1\right) \le 1.
  \end{equation*}
  Pulling out the known factor and using the properties of conditional expectations,
  shows
  \begin{equation*}
    \frac{1}{\Phi^{N-1}}\bbE\left(\Phi^N|\Phi^{N-1},\dots,\Phi^1\right) \le 1.
  \end{equation*}

  Because $K$ is a stopping time with finite expectation for ${\{S^k\}}_{k \in \bbN}$
  and because we can define ${\{T^k\}}_{k \in \bbN}$ in terms of ${\{S^k\}}_{k \in \bbN}$, and
  then we can define ${\{\Phi^k\}}_{k \in \bbN}$ in terms of ${\{T^k\}}_{k \in \bbN}$, so $K$ is
  a stopping time for ${\{\Phi^k\}}_{k \in \bbN}$ with finite expectation.
  Because ${\{\Phi^k\}}_{k \in \bbN}$ is positive,
  we can use Doob's optional stopping theorem to conclude that
  \begin{equation*}
    \bbE(\Phi^K) \le \bbE(\Phi^0).
  \end{equation*}
  The last part of the proof requires us to use Markov's inequality for $\Phi^K$,
  yielding
  \begin{equation}\label{eq:chernoff markov bound}
    \bbP(\Phi^K \ge 1) \le \bbE(\phi^t) \le \bbE(\Phi^0) = e^{-\gamma^2 E}.
  \end{equation}
  It is a classic result that if $\gamma$ is sufficiently small, for any
  $x \in [-\gamma, \gamma]$, $\frac{x}{1 + x} \le \log(1 + x)$, so we can conclude
  \begin{align*}
    \frac{1}{1 + \gamma}
    \sum_{k = 0}^{K - 1}(1 - T^k) - \frac{1}{1 - \gamma}\sum_{k = 0}^{K - 1} T^k
    &= \frac{1}{\gamma}\left(\frac{\gamma}{1 + \gamma} \sum_{k = 0}^{K - 1}(1 - T^k)
    + \frac{-\gamma}{1 - \gamma}\sum_{k = 0}^{K - 1} T^k\right) \nonumber \\
    &\le \frac{\log(1 + \gamma)}{\gamma} \sum_{k = 0}^{K - 1} (1 - T^k) +
      \frac{\log(1 - \gamma)}{\gamma}\sum_{k = 0}^{K - 1} T^k.
  \end{align*}
  This bound transforms into a relation between probabilities, i.e.
  \begin{align*}
    &\bbP\left(\frac{1}{1 + \gamma} \sum_{k = 0}^{K - 1} (1 - T^k)
    - \frac{1}{1 - \gamma}\sum_{k = 0}^{K - 1} T^k \ge \gamma E \right) \\
    &\quad\le \bbP\left(\log(1 + \gamma) \sum_{k = 0}^{K - 1} (1 - T^k) +
      \log(1 - \gamma)\sum_{k = 0}^{K - 1} T^k \ge \gamma^2 E\right).
  \end{align*}
  The exponential is monotone, so
  \begin{align}\label{eq:chernoff i am tired 1}
    \bbP\left(\frac{1}{1 + \gamma} \sum_{k = 0}^{K - 1} T^k
    - \frac{1}{1 - \gamma}\sum_{k = 0}^{K - 1} (1 - T^k) \ge \gamma E \right)
    \le \bbP\left({(1 + \gamma)}^{\sum_{k = 0}^{K - 1} T^k}
      {(1 - \gamma)}^{\sum_{k = 0}^{K - 1} (1 - T^k)} \ge e^{\gamma^2 E}\right),
  \end{align}
  and using the definition of $\Phi^T$,
  \begin{equation*}
    \bbP\left({(1 + \gamma)}^{\sum_{k = 0}^{K - 1} (1 - T^k)}
      {(1 - \gamma)}^{\sum_{k = 0}^{K - 1} T^k} \ge e^{\gamma^2 E}\right)
    = \bbP(\Phi^T \ge 1).
  \end{equation*}
  Combining~\eqref{eq:chernoff i am tired 1} with~\eqref{eq:chernoff markov bound}
  gives
  \begin{equation}\label{eq:chernoff whhhhhhy?}
    \bbP\left(\frac{1}{1 + \gamma} \sum_{k = 0}^{K - 1} (1 - T^k)
    - \frac{1}{1 - \gamma}\sum_{k = 0}^{K - 1} T^k \ge \gamma E \right)  \le e^{-\gamma^2 E}.
  \end{equation}

  In order to complete the proof, we investigate the expression
  from the left hand side of~\eqref{eq:chernoff whhhhhhy?}.
  By our definition $S^K = \sum_{k = 0}^{K - 1}T^K$,
  so $K - S^K = \sum_{k = 0}^{K - 1}(1 - T^K)$. This allows us to
  rewrite~\eqref{eq:chernoff whhhhhhy?}, yielding
  \begin{align*}
    \bbP\left(\frac{1}{1 + \gamma}(K - S^K) - \frac{1}{1 - \gamma}S^k \ge \gamma E \right)
    &=\bbP((1 - \gamma)(K - S^K) - (1 + \gamma)S^K \ge \gamma(1 - \gamma^2)E) \nonumber \\
    &=\bbP((1 - \gamma)K \ge 2 S^K + \gamma(1 - \gamma^2)E) \le e^{-\gamma^2 E}, \nonumber \\
  \end{align*}
  completing the proof.
\end{proof}
\begin{corollary}\label{corr: chernoff bound for stopping times}
  In the setting of Proposition~\ref{prop: chernoff bound for stopping times},
  set $E = \bbE(K)$. Then
  \begin{equation*}
    \bbP((1 - \gamma)K \ge 2 S^K + \gamma(1 - \gamma^2)\bbE(K)) \le e^{-\gamma^2 \bbE(K)}.
  \end{equation*}
\end{corollary}

This result is a generalization of the standard Chernoff bounds.
\begin{proposition}[Chernoff Bound]%
  \label{prop: markov perturbation inequality}
  Let ${\{T^k\}}_{k \in \bbN}$ be a sequence of independent Bernoulli random variables,
  with $\bbP(T^k = 1) \ge (1 - \delta)$ for all $k \in \bbN$. Then for any $\gamma \in (0, 1)$ and for
  any $n \in \bbN$
  \begin{equation*}
    \bbP\left(\sum_{k = 0}^{n-1} T^k \le (1 - \gamma) \bbE\left(\sum_{k = 0}^{n-1} T^k\right) \right)
      \le e^{-n(1 - \delta)\frac{\gamma^2}{2}}.
  \end{equation*}
\end{proposition}

\section{Regularity}
The notion of smoothness that emerged as the key ingredient in the analysis
of non-smooth Newton-type methods is that of Newton differentiability. This
notion has its roots in the work on semismooth optimization and it was
introduced by Qi in~\cite{Qi_96Cdif}, where it was called C-differentiability.
Further works, such as~\cite{BroUlb22Newt}, where it was used in the context of
infinite dimensional optimization, renamed this notion to Newton differentiability.

\begin{definition}[Weak Uniform Newton differentiability]%
\label{d:uniform Newton Diff}
  A function $F:U \subset \bbRn \to \bbRn$ is called
  {\em weakly uniformly Newton differentiable on $U$\/} if there exists a set
  valued mapping $\mathcal{H}F:V \setto \bbRnxn$ and $c > 0$ such that for all $x \in U$ and all
  $y \in U$,
  \begin{equation*}
    \sup_{H \in \mathcal{H}F(x)}\frac{\|F(x) - F(y) - H(x - y)\|}{\|x - y\|} \le c.
  \end{equation*}
  The smallest such number $c$ is called the {\em constant of Newton
    differentiability}.
\end{definition}

In order to replace the (Clarke) subdifferential with a
single valued object that can equally characterize first order stationarity,
we require the following adaptation.
\begin{definition}[Single-Valued Adaptation]
  Let $F:U \setto \bbRn$. Consider the auxiliary function $F_1^m: U \setto \bbRn$, defined by
  \begin{equation*}
    F_1^m(x) = \proj_{\overline{F(x)}}0.
  \end{equation*}
  Any selection $F_1:U \to \bbRn$ of $F_1^m$ is called a {\em single-valued adaptation
  of $F$}.
  Clearly $0 \in F(x)$ if and only if $0 = F_1(x)$, so the problem of finding
  a zero of $F$ is the same as the problem of finding a zero of $F_1$.
\end{definition}

For the remaining of this work, we impose the following assumption on the
regularity of the objective.
\begin{assumption}[Regularity Assumption]\label{assum:reg}
  Let $f: U \subseteq \bbRn \to \bbR$ be Lipschitz continuous and assume $U$ is small enough and
  with an isolated local minimum and unique critical point
  at $\xbar$ and denote by $\partial f_1$ the single-valued adaptation of the Clarke
  subdifferential of $f$. Assume that $\partial f_1$ is uniformly weakly Newton
  differentiable on $U$. Denote the Newton differential of $\partial f_1$ by
  $\mathcal{H}f$ and assume that $\forall x \in U$ all $H \in \mathcal{H}f(x)$ are positive definite and
  further assume that the set $\bigcup_{x \in U}\{\|{{H}^{-1}}\|~|~H \in \mathcal{H}f(x)\}$ is bounded by
  $\Omega < \infty$ and the set $\bigcup_{x \in U}\{\|H\|~|~H \in \mathcal{H}f(x)\}$ is bounded by $\omega < \infty$.
  Let $c$ be as in Definition~\ref{d:uniform Newton Diff} and assume that
  $c\Omega < 1$.
\end{assumption}

This regularity assumption can be used to construct a Newton-like sequence with strong
convergence guarantees. In our work, we focus on the notions
of sufficient decrease and subdifferential lower bound. These notions
have been looked at by Bolte et al.\ in~\cite{BolSabTeb14Prox}.
\begin{lemma}[Sufficient Decrease]\label{lemma:sufficient decrease}
  If Assumption~\ref{assum:reg} is satisfied,
  then there exists $\rho > 0$ such that for all $x \in U$ and for all
  $x^{+} \in \{x - H^{-1}\partial f_1 (x)~|~H \in \mathcal{H}f(x), \det H \ne 0\} \cap U$,
  \begin{equation*}
    f(x) - f(x^{+}) \ge \rho\|x^{+} - x\|^2.
  \end{equation*}
\end{lemma}
\begin{proof}
  Let $x^{+} \in \{x - H^{-1}\partial f_1 (x)~|~H \in \mathcal{H}f(x), \det H \ne 0\} \cap U$ and set
  $H \in \mathcal{H}f(x)$ such that $H(x^{+} - x) = \partial f_1(x)$.
  In order to prove this lemma, we require two model decrease bounds. Pursuant
  to this, we introduce the model
  \begin{equation*}
    Q(x) = f(x) + \langle \partial_1 f(x), x - x \rangle + \frac{1}{2} \langle x - x, H (x - x)\rangle
  \end{equation*}

  The first bound we need is
  \begin{align}\label{eq:sufficient decrease monster}
    Q&(x^{+}) - f(x^{+})
       = f(x) - f(x^{+}) + \langle \partial_1 f(x), x^{+} - x \rangle
       + \frac{1}{2} \langle x^{+} - x, H (x^{+} - x) \rangle \nonumber \\
     &= \int_0^1 \langle -\partial_1 f(x + t(x^{+} - x)), x^{+} - x \rangle  \diff t+
       \langle \partial_1 f(x), x^{+} - x \rangle \nonumber \\
     &\quad + \frac{1}{2} \langle x^{+} - x, H (x^{+} - x) \rangle \nonumber \\
     &= \int_0^1 \langle \partial_1 f(x)-\partial_1 f(x + t(x^{+} - x)),
       x^{+} - x \rangle  \diff t \nonumber \\
     &\quad + \frac{1}{2} \langle x^{+} - x, H (x^{+} - x) \rangle \nonumber \\
     &= \int_0^1 \langle \partial_1 f(x)-\partial_1 f(x + t(x^{+} - x)),
       x^{+} - x \rangle  \diff t \nonumber \\
     &\quad + \int_0^1 t \diff t \langle x^{+} - x, H (x^{+} - x) \rangle \nonumber \\
     &= \int_0^1 \langle \partial_1 f(x)-\partial_1 f(x + t(x^{+} - x)), x^{+} - x \rangle \nonumber \\
     &\quad + \langle t(x^{+} - x), H (x^{+} - x) \rangle \diff t \nonumber \\
     &\le \int_0^1 \|\partial_1 f(x)-\partial_1 f(x + t(x^{+} - x)) + H t (x^{+} - x) \|
       \|x^{+} - x\| \diff t \nonumber \\
     &\le \int_0^1 tc\|(x^{+} - x) \|\|x^{+} - x\| \diff t \nonumber \\
     &=\frac{c}{2}\|x^{+} - x\|^2.
  \end{align}
  The second bound is derived from the observation that
  $x^{+} - x = - {H}^{-1}\partial_1 f(x)$, so
  \begin{align*}
    Q_k(x^{+})
    &- f(x) = \langle \partial_1 f(x), x^{+} - x \rangle
    + \frac{1}{2} \langle x^{+} - x, H (x^{+} - x) \rangle \nonumber \\
    &= -\langle \partial_1 f(x), {H}^{-1}\partial_1 f(x)  \rangle
    + \frac{1}{2} \langle {H}^{-1}\partial_1 f(x) , H {H}^{-1}\partial_1 f(x)  \rangle \nonumber \\
    &=-\frac{1}{2} \langle \partial_1 f(x), {H}^{-1}\partial_1 f(x) \rangle,
  \end{align*}
  where we have used the fact that $H$ is symmetric. Equivalently, and
  using the positive definiteness of $H$
  \begin{equation}\label{eq: sufficient decrease not monster}
    f(x) - Q_k(x^{+})
    = \frac{1}{2} \langle \partial_1 f(x), {H}^{-1}\partial_1 f(x) \rangle \ge 0.
  \end{equation}

  After a simple algebraic manipulation, we see
  \begin{equation*}
    1 - \frac{f(x) - f(x^{+})}{f(x) - Q(x)} =
    \frac{f(x^{+}) - Q(x)}{f(x) - Q(x)}.
  \end{equation*}

  Substituting~\eqref{eq:sufficient decrease monster}
  and~\eqref{eq: sufficient decrease not monster} in the previous bound yields
  \begin{align}\label{eq:sufficient decrease getting close to the end 2}
    1 - \frac{f(x) - f(x^{+})}{f(x) - Q(x)} \le
    c\frac{\|x^{+} - x\|^2}{\langle \partial_1 f(x), {H}^{-1}\partial_1 f(x) \rangle}.
  \end{align}

  The final hurdle consists in computing
  \begin{align}\label{eq:sufficient decrease why this}
    \|x^{+} - x\|^2 &= \langle {H}^{-1}\partial_1 f(x), {H}^{-1}\partial_1 f(x) \rangle
    \nonumber \\
    &\le \|{H}^{-1}\||\langle \partial_1 f(x), {H}^{-1}\partial_1 f(x) \rangle| \nonumber \\
    &= \|{H}^{-1}\|\langle \partial_1 f(x), {H}^{-1}\partial_1 f(x) \rangle \nonumber \\
    &\le \Omega\langle \partial_1 f(x), {H}^{-1}\partial_1 f(x) \rangle
  \end{align}
  and equivalently
  \begin{equation}\label{eq:sufficient decrease getting close to the end}
    \frac{\|x^{+} - x\|^2}{\langle \partial_1 f(x), {H}^{-1}\partial_1 f(x) \rangle} \le \Omega.
  \end{equation}
  Combining~\eqref{eq:sufficient decrease getting close to the end 2}
  with~\eqref{eq:sufficient decrease getting close to the end} and using
  the assumption $c\Omega < 1$, we see that
  \begin{equation*}
    1 - \frac{f(x) - f(x^{+})}{f(x) - Q(x)} \le 1 - (1 - c\Omega),
  \end{equation*}
  yielding
  \begin{equation*}
    f(x) - f(x^{+}) \ge (1 - c)\Omega(f(x) - Q(x)).
  \end{equation*}
  Finally, using~\eqref{eq: sufficient decrease not monster}
  and~\eqref{eq:sufficient decrease why this} produces the desired result
  \begin{equation*}
    f(x) - f(x^{+}) \ge \frac{1 - c}{2}\|x^{+} - x\|^2.
  \end{equation*}
\end{proof}

\begin{lemma}[Subdifferential Lower Bound]\label{lemma: gradient lower bound}
  If Assumption~\ref{assum:reg} is satisfied,
  then there exists $\tau \in (0, \infty)$ such that for all $x \in U$ and for all
  $x^{+} \in \{x - H^{-1}\partial f_1 (x)~|~H \in \mathcal{H}f(x), \det H \ne 0\} \cap U$,
  \begin{equation*}
    \|\partial_1 f(x)\| \le \tau\|x^{+} - x\|.
  \end{equation*}
\end{lemma}
\begin{proof}
  Let $x^{+} \in \{x - H^{-1}\partial f_1 (x)~|~H \in \mathcal{H}f(x), \det H \ne 0\} \cap U$ and set
  $H \in \mathcal{H}f(x)$ such that $H(x^{+} - x) = \partial f_1(x)$, so
  \begin{equation*}
    \|x^{+} - x\| = \|{H}^{-1} \partial_1f(x)\|
  \end{equation*}
  and because of the bound on $\|H\|$
  \begin{equation*}
    \|\partial_1f(x)\| = \|H{H}^{-1}\partial_1f(x)\| \le \|H^{k}\|\|{H^k}^{-1} \partial_1f(x)\| \le
    \omega \|{H}^{-1} \partial_1f(x)\|,
  \end{equation*}
  we can derive the conclusion
  \begin{equation*}
    \|x^{+} - x\| \ge \frac{1}{\omega} \|\partial_1f(x)\|.
  \end{equation*}
\end{proof}

\section{Stochastic Framework}
In this section we present an algorithmic framework for solving the problem
\begin{equation*}
  \min_{x \in U} f(x),
\end{equation*}
where $f$ satisfies Assumption~\ref{assum:reg}, but the Newton differential,
$\mathcal{H}f$ is only available through a stochastic oracle. In this context this, the
minimization problem is equivalent to the problem of finding the zero of $\partial f_1$,
i.e. $\xbar$.

Our proposed algorithm works by sampling the oracle, represented by a random variable
concentrated on the Newton
differential, $\mathcal{H}f$. The stochastic framework employs a backtracking approach in
order to check if an instance of the random variable is part of the Newton
differential. This insight is made concrete by
Algorithm~\ref{alg: Stochastic Newton}.
\begin{algorithm}
  \caption{\label{alg: Stochastic Newton} Stochastic Newton-type Method}
  \KwData{$f, \partial_1 f, x^0, c_0 \in \lbrack0, \infty\rparan , \alpha \in (0, 1), \varepsilon \in (0, 1)$\;}
  $k \gets 0$\;
  \While{$\|\partial_1 f(x^k)\| \ge \varepsilon$}{
    \KwSample{$B^k$\;}
    $y \gets x^k - B^k \partial_1 f(x^k)$\;
    \eIf{$f(x^k) - f(y) \ge c\|y - x^k\|$ {\bf and} $\|\partial_1 f(x^k)\| \le c^k\|y - x^k\|$}{
      $x^{k+1} \gets y$\;
      $c_{k+1} \gets c_k$\;
    }{
      $x^{k+1} \gets x^k$\;
      $c_{k+1} \gets \alpha c_k$\;
    }
    $k \gets k+1$\;
  }
  \Return{$x^{k}$};
\end{algorithm}
\begin{remark}
  In Algorithm~\ref{alg: Stochastic Newton}, if an iteration executes the
  {\bf then} branch of
  the conditional, we call this iteration a {\em successful iteration}. Otherwise
  we call it an {\em unsuccessful iteration}. Similarly, if ${B^k}^{-1} \in \mathcal{H}f(x^k)$
  we call the iteration {\em true}, and otherwise {\em false}.
\end{remark}

Based on the discussion from the previous section, the number of true,
but unsuccessful iterations is going to be bounded. This guarantees that all
successful iterations make some progress towards the critical point of $f$ and
all true iterations are going to be successful.

\begin{theorem}\label{thm:stochastic newton}
  Assume Assumption~\ref{assum:reg} is satisfied and let
  $\bbH:U \to \{A \in \bbRnxn~|~\det A \ne 0\}$ be a random function such that for any $x \in U$
  \begin{equation}\label{eq:stochastic newton probability bound}
     \bbP(\bbH(x) \in \mathcal{H}f(x)) \ge 1 - \delta.
  \end{equation}
  Consider for any $x^0 \in U$ the sequence
  generated by Algorithm~\ref{alg: Stochastic Newton} with $X^0 \sim \delta_{x^0}$ and
  ${B^k}^{-1} \sim \bbH(X^k)$ random variables. Further assume that for
  all $k$, $X^k \in U \aseq$.
 Let $K$ be the
 (possibly infinite) random variable representing the number of iterations
 before the algorithm terminates.
  \begin{enumerate}
  \item
    Then there exists $\rho > 0$ and $\tau > 0$
    such that
    \begin{equation*}
      \bbE(K) \le \frac{1}{1 - \delta}\left(
        \frac{f(x^0) - f(\xbar)}{\rho{(\tau\varepsilon)}^2}
        + \frac{\log \frac{\rho}{c^0}}{\log \alpha }
      \right)
    \end{equation*}
    Furthermore, for any $\gamma \in (0, 1)$
    \begin{equation*}
      \bbP\left(K \ge \frac{2 + \gamma(1 - \gamma^2)}{1 - \gamma} \frac{1}{1 - \delta} \left(
          \frac{f(x^0) - f(\xbar)}{\rho{(\tau\varepsilon)}^2}
          + \frac{\log \frac{\rho}{c_0}}{\log \alpha} \right)\right)
      \le e^{\frac{-\gamma^2}{1 - \delta}\left(
          \frac{f(x^0) - f(\xbar)}{\rho{(\tau\varepsilon)}^2}
          + \frac{\log \frac{\rho}{c_0}}{\log \alpha } \right)}.
    \end{equation*}
  \item
    Then there exist $\rho > 0$ and $\tau > 0$, such that for all $\gamma \in (0, 1)$,
    \begin{equation*}
      \bbP\left(\min_{n \le K}
        \|\partial_1 f(x^n)\| \ge \sqrt{\frac{f(x^0) - f(\xbar)}{\rho\tau^2
            \left((1 - \gamma)K - \frac{\log \frac{\rho}{c_0}}{\log \alpha } \right)}} \right)
      \le e^{-K(1 - \delta)\frac{\gamma^2}{2}}.
    \end{equation*}
  \end{enumerate}
\end{theorem}
\begin{proof}\phantom{Space}
  \begin{enumerate}
  \item
    From Lemma~\ref{lemma:sufficient decrease} there exists $\rho > 0$ such that
    for any $x \in U$ and for all $H \in \mathcal{H}f(x)$
    \begin{equation}\label{eq:stochastic newton suff d}
      f(x) - f(x - H^{-1} \partial_1 f(x)) \ge \rho\|H^{-1} \partial_1 f(x)\|,
    \end{equation}
    and from Lemma~\ref{lemma: gradient lower bound} we know that there exists
    $\tau > 0$ such that for any $x \in U$ and for all $H \in \mathcal{H}f(x)$
    \begin{equation}\label{eq:stochastic newton glb}
      \|H^{-1} \partial_1 f(x)\| \ge \tau\|\partial_1 f(x)\|.
    \end{equation}

    We consider the sequence ${\{x^k\}}_{k \in \bbN}$ produced by the successful iterates of
    the algorithm.
    We call an iteration $k$ true if $B^k \in \{H^{-1}~|~H \in \mathcal{H}f(x^k)\}$ and define
    the Bernoulli random variable $T^k$ to be 1 if the $k$th iteration is
    true and 0 otherwise.

    From~\eqref{eq:stochastic newton probability bound} we know that for any
    iteration $k$
    \begin{align*}
      \bbP(T^k = 1|X^k) &= \bbP({B^k}^{-1} \in \mathcal{H}f(X^k)|X^k) \nonumber \\
                     &= \bbP(\bbH(X^k) \in \mathcal{H}f(X^k)|X^k) \nonumber \\
                     &\ge 1 - \delta.
    \end{align*}
    Because $T^k$ are Bernoulli,
    \begin{equation*}
      \bbE(T^k|X^k)  = \bbP(T^k = 1|X^k) \ge 1 - \delta.
    \end{equation*}
    Using the law of total expectation,
    \begin{equation*}
      \bbE(T^k) = \bbE(\bbE(T^k|X^k)) \ge \bbE(1 - \delta) = 1 - \delta.
    \end{equation*}

    We can conclude from~\eqref{eq:stochastic newton suff d}
    and~\eqref{eq:stochastic newton glb} that for any true and successful iteration
    $k \in \bbRn$ it holds that
    \begin{equation*}
      f(x^k) - f(x^{k+1}) \ge \rho\|x^{k+1} - x^k\|^2
    \end{equation*}
    and
    \begin{equation*}
      \|x^{k + 1} - x^{k}\| \ge \tau\|\partial_1 f(x^k)\|.
    \end{equation*}
    For any successful but false iteration, clearly
    \begin{equation}\label{eq:new proof for remark 1}
      f(x^k) - f(x^{k+1}) \ge 0
    \end{equation}
    and
    \begin{equation}\label{eq:new proof for remark 2}
      \|x^{k + 1} - x^{k}\| \ge 0.
    \end{equation}
    We can unify these two bounds in
    \begin{equation}\label{eq:new proof combined equation 1}
      f(x^k) - f(x^{k+1}) \ge T^k\rho\|x^{k+1} - x^k\|^2
    \end{equation}
    and
    \begin{equation}\label{eq:new proof combined equation}
      \|x^{k + 1} - x^{k}\| \ge T^k\tau\|\partial_1 f(x^k)\|.
    \end{equation}
    for any successful iteration.

    \begin{remark}
      The bounds in~\eqref{eq:new proof for remark 1}
      and~\eqref{eq:new proof for remark 2} can easily be improved to
      involve $\alpha^{\text{number of unsuccessful iterations}}c_0$ instead of 0, yielding a
      slightly tighter result.
    \end{remark}

    Denote by $T_S$ the random number of true successful iterations and denote
    by $S$ the total number of successful iterations. We can assume that
    the algorithm did not terminate before the $S$th iteration, so
    \begin{equation}\label{eq:new proof for paper bound on stochastic}
      \min_{n \le S} \|\partial_1 f(x^n)\| > \varepsilon.
    \end{equation}
    Clearly, for any iteration $k$ one has
    \begin{equation*}
      T^k\tau \|\partial_1 f(x^k)\| \ge T^k\tau \min_{n \le S} \|\partial_1 f(x^n)\|,
    \end{equation*}
    and using~\ref{eq:new proof combined equation}
    \begin{equation*}
      \|x^{k + 1} - x^{k}\| \ge T^k\tau \min_{n \le S} \|\partial_1 f(x^n)\|.
    \end{equation*}
    Squaring everything and multiplying by $T^k\rho$, and using the fact that
    $T^k$ is Bernoulli, and as such ${T^k}^2 = T^k$, we obtain
    \begin{equation*}
      T^k\rho{\|x^{k + 1} - x^{k}\|}^2 \ge T^k \rho \tau^2{\min_{n \le S} \|\partial_1 f(x^n)\|}^2.
    \end{equation*}
    Using~\ref{eq:new proof combined equation 1} shows
    \begin{equation*}
      f(x^k) - f(x^{k+1}) \ge T^k \rho \tau^2{\min_{n \le S} \|\partial_1 f(x^n)\|}^2,
    \end{equation*}
    followed by summing and telescoping, to arrive at
    \begin{equation}\label{eq:new proof this will appear again}
      f(x^0) - f(x^{S+1}) \ge \sum_{k = 0}^{S}T^k \rho \tau^2{\min_{n \le S} \|\partial_1 f(x^n)\|}^2.
    \end{equation}
    To complete this part of the proof, we use the fact that $\xbar$ is the minimum of $f$,
    so $f(\xbar) \le f(x^{S+1})$ and the fact that $T^k$ for $k \in \bbN$ are Bernoulli variables
    indicating if a successful iteration is true, so their sum counts the number
    of true and successful iterations, together
    with~\ref{eq:new proof for paper bound on stochastic}, yielding
    \begin{equation}\label{eq:stochastic bound t1s}
      T_S = \sum_{k = 0}^{S}T^k \le \frac{f(x^0) - f(\xbar)}{\rho{(\tau\varepsilon)}^2}.
    \end{equation}

    Next we focus on the number of unsuccessful iterations. For this
    purpose, now consider the entire, sequence ${\{x^k\}}_{k \in \bbN}$ produced by
    the algorithm, including all the repeated points yielded by unsuccessful
    iterations. As before, we denote with $T^k$ the random variable representing
    whether an iteration is true or not, and by $I$ and $T_I$ the total number of
    unsuccessful, and true and unsuccessful iterations respectively. We remark
    that $T_I \le I$.

    Clearly, when $c_k \le \min{\rho, \tau}$ we
    can use Lemma~\ref{lemma:sufficient decrease} and
    Lemma~\ref{lemma: gradient lower bound} to see that a true iteration is
    successful. Any unsuccessful iteration decreases $c_k$ by a factor of $\alpha$.
    This provides the bound
    \begin{equation*}
      \alpha^{I}c_0 \ge \rho
    \end{equation*}
    and rearranging
    \begin{equation}\label{eq:stochcastic bound tu}
      T_I \le I \le \frac{\log \frac{\rho}{c_0}}{\log \alpha}.
    \end{equation}

    We can easily see that $T_S + T_I = T$. By the definition of $T^k$, we know
    \begin{equation}\label{eq:stochastic newton def of T}
      T = \sum_{k = 0}^{K-1} T^k.
    \end{equation}
    From the two bounds~\eqref{eq:stochastic bound t1s}
    and~\eqref{eq:stochcastic bound tu} we can derive
    \begin{equation}\label{eq:stochastic bound t}
      T \le \frac{f(x^0) - f(\xbar)}{\rho{(\tau\varepsilon)}^2}
      + \frac{\log \frac{\rho}{c_0}}{\log \alpha}.
    \end{equation}
    This allows us to express $K$ as
    \begin{equation*}
      K = \min \left \{N \in \bbN~\left|~ \sum_{k = 0}^{K - 1} T^k
        \le \frac{f(x^0) - f(\xbar)}{\rho{(\tau\varepsilon)}^2}
        + \frac{\log \frac{\rho}{c_0}}{\log \alpha} \right.\right \}
    \end{equation*}
    and to conclude, that $K$ is a hitting time for the stochastic process
    and thus $K$ is a stopping time
    with finite expectation for the stochastic process. Next we
    use~\eqref{eq:stochastic newton def of T} to compute the expectation of
    $T$
    \begin{equation*}
      \bbE(T) = \bbE\left(\sum_{k = 0}^{K - 1} T^k\right),
    \end{equation*}
    which can be bounded by employing Proposition~\ref{prop:expectation of sums}
    \begin{equation*}
      \bbE(T) \le \bbE(K)(1 - \delta).
    \end{equation*}

    Rearranging in~\eqref{eq:stochastic bound t}
    and using Proposition~\ref{prop:expectation of sums}
    yields
    \begin{equation}\label{eq:stochastic bound on e(K)}
      \bbE(K) \le \frac{\bbE(T)}{1 - \delta} = \frac{\bbE(T_S + T_I)}{1 - \delta} \le
      \frac{1}{1 - \delta}\left(
        \frac{f(x^0) - f(\xbar)}{\rho{(\tau\varepsilon)}^2}
        + \frac{\log \frac{\rho}{c_0}}{\log \alpha}\right),
    \end{equation}
    finishing the first part of the proof.

    For the second part, we use
    Corollary~\ref{corr: chernoff bound for stopping times}
    to conclude that, for any $\gamma \in (0,1)$
    \begin{equation*}
      \bbP((1 - \gamma)K \ge 2T + \gamma(1 - \gamma^2)\bbE(K)) \le e^{-\gamma^2 \bbE(K)}.
    \end{equation*}
    Using~\eqref{eq:stochastic bound t} together
    with~\eqref{eq:stochastic bound on e(K)}
    \begin{equation*}
      \frac{2}{1 - \gamma}T + \gamma(1 + \gamma)\bbE(K) \le
      \frac{2 + \gamma(1 - \gamma^2)}{(1 - \gamma)(1 - \delta)}\left(
        \frac{f(x^0) - f(\xbar)}{{(\tau\varepsilon)}^2}
        + \frac{\log \frac{\rho}{c_0}}{\log \alpha} \right),
    \end{equation*}
    and this allows us to relate the probabilities
    \begin{align*}
      \bbP&\left(K \ge \frac{2 + \gamma(1 - \gamma^2)}{(1 - \gamma)(1 - \delta)} \left(
         \frac{f(x^0) - f(\xbar)}{\rho{(\tau\varepsilon)}^2}
         + \frac{\log \frac{\rho}{c_0}}{\log \alpha} \right)\right) \nonumber \\
       &\le \bbP((1 - \gamma)K \ge 2T + \gamma(1 - \gamma^2)\bbE(K)) \nonumber \\
       &\le e^{-\gamma^2 \bbE(K)}.
    \end{align*}
    From~\eqref{eq:stochastic bound on e(K)}, we compute
    \begin{equation*}
      e^{-\gamma^2 \bbE(K)} \le e^{-\gamma^2 \frac{1}{1 - \delta}\left(
          \frac{f(x^0) - f(\xbar)}{\rho{(\tau\varepsilon)}^2}
          + \frac{\log \frac{\rho}{c_0}}{\log \alpha}\right)},
    \end{equation*}
    finalizing the argument.

  \item
    We consider again the sequence produced by the successful iterations
    and we use the same notation as before.
    From Proposition~\ref{prop: markov perturbation inequality}, we bound
    \begin{equation}\label{eq: proof stochastic 2 first hing}
      \bbP(T \le (1 - \gamma)K) \le e^{-K(1 - \delta)\frac{\gamma^2}{2}}.
    \end{equation}
    Like in the proof of the previous part, we
    can bound the number of true but unsuccessfully iterations by
    \begin{equation*}
      T_I \le \frac{\log \frac{\rho}{c_0}}{\log \alpha },
    \end{equation*}
    so
    \begin{equation*}
      T_S \ge T - \frac{\log \frac{\rho}{c_0}}{\log \alpha }.
    \end{equation*}
    Using the probability from~\eqref{eq: proof stochastic 2 first hing}
    \begin{equation}\label{eq: proof stochastic 2 second things}
      \bbP\left(T_S \le (1 - \gamma)K - \frac{\log \frac{\rho}{\rho_0}}{\log \alpha } \right)
      \le  \bbP(T \le (1 - \gamma)K) \le e^{-K(1 - \delta)\frac{\gamma^2}{2}}.
    \end{equation}

    Recalling from~\ref{eq:new proof this will appear again}, together
    with the definition of $T_S$, one obtains that
    \begin{equation*}
      f(x^0) - f(\xbar) \ge T_S \rho \tau^2{\min_{n \le S} \|\partial_1 f(x^n)\|}^2.
    \end{equation*}
    Rearranging, gives
    \begin{equation}\label{eq:new proof bound on min nabla}
      \min_{n \le S} \|\partial_1 f(x^n)\| \le \sqrt{\frac{f(x^0) - f(\xbar)}{T_S \rho\tau^2}}.
    \end{equation}

    Using~\ref{eq:new proof bound on min nabla}, we can conclude that the event
    \begin{equation*}
      \sqrt{\frac{f(x^0) - f(\xbar)}{\rho\tau^2
          \left((1 - \gamma)K - \frac{\log \frac{\rho}{c_0}}{\log \alpha } \right)}}
      \le \min_{n \le S} \|\partial_1 f(x^n)\|
    \end{equation*}
    implies the event
    \begin{equation*}
      \sqrt{\frac{f(x^0) - f(\xbar)}{\rho\tau^2
          \left((1 - \gamma)K - \frac{\log \frac{\rho}{c_0}}{\log \alpha } \right)}}
      \le \sqrt{\frac{f(x^0) - f(\xbar)}{T_S \rho\tau^2}}.
    \end{equation*}
    Then
    \begin{equation*}
      T_S \le (1 - \gamma)K - \frac{\log \frac{\rho}{c_0}}{\log \alpha }.
    \end{equation*}
    This induces the probability relation
    \begin{align*}
      \bbP&\left(\min_{n \le T_S}
      \|\partial_1 f(x^n)\| \ge \sqrt{\frac{f(x^0) - f(\xbar)}{\rho\tau^2
        \left((1 - \gamma)K - \frac{\log \frac{\rho}{c_0}}{\log \alpha } \right)}} \right)\\
      &\quad\le \bbP\left(T_S \le (1 - \gamma)K - \frac{\log \frac{\rho}{c_0}}{\log \alpha }\right).
    \end{align*}
    Substitution in~\eqref{eq: proof stochastic 2 second things} yields
    \begin{equation*}
      \bbP\left(\min_{n \le T_S}
        \|\partial_1 f(x^n)\| \ge \sqrt{\frac{f(x^0) - f(\xbar)}{\rho_{\min}d^2
            \left((1 - \gamma)K - \frac{\log \frac{\rho_{\min}}{\rho^0}}{\log \alpha } \right)}} \right)
      \le e^{-K(1 - \delta)\frac{\gamma^2}{2}},
    \end{equation*}
    finishing the proof.
  \end{enumerate}
\end{proof}

The first part of Theorem~\ref{thm:stochastic newton} provides a probability
distribution on the number of iterates required in order to achieve approximate
first order optimality. The second part provides a dual result on the probability
distribution of the first order residual after a fixed amount of iterations.

\begin{remark}
  Different to a significant part of the literature, see for
  example~\cite{YanMilWenZha22Asto,Tur21Thec,ByrHanNocSin16Asto,GuiRob24Hand,
    BotCurNoc18Opti,BolByrNoc18Exac},
  we {\bf do not require}
  that the estimator is unbiased, i.e. $\bbE(\bbH(x)) \in \mathcal{H}f(x)$, or that it has finite
  variance, i.e $\bbE(\tr {\bbH(x)}^2) < \infty$. Also, the sub-exponential tail condition
  encountered in~\cite{NaDerMah23Hess} is not present in our work.
\end{remark}

\section{Applications}
In this section we showcase the wide applicability of our framework. We first
present a theoretical result covering random additive noise. Then we
cover the numerical analysis of a quasi-Newton method involved in a physical
experiment from
X-FEL imaging. The relationship between Newton differentiability and
quasi-Newton methods has further been explored in~\cite{Pin24Newt}.
The final example provides a theoretical analysis of a sketching
approach to Newton-type methods, together with numerical results.
The advantage of our approach when compared to previous methods, such as the
ones developed in~\cite{YanMilWenZha22Asto,MilXiaWenUlb22Onth,
  MilXIaCenWenUlb19Asto}, lies in
the applicability of our strong theoretical results to a large class of
algorithms and problems.

\subsection{Random Noise}
Regarding the algebraic setup for this subsection, we consider the Frobenius norm
of a matrix and we denote the closed ball in this norm with radius $\rho$ and
center $M \in \bbRn$ by $\bbB_\rho[M]$. We recall that the Frobenius norm is given
by an inner product, i.e. $\|M\|^2 = \tr M^T M$ and that this inner product
coincides with the usual inner product if we identify $\bbRnxn$ with $\bbR^{n^2}$.
We denote this identification with $\nu:\bbRnxn \to \bbR^{n^2}$ and remark that
$\nu(M){\nu(M)}^T \in \bbR^{n^2 \times n^2}$ and $\tr \nu(M){\nu(M)}^T = \|M\|^2$.

The purpose of this subsection is to work thorough an
example of a Newton-type method where the Newton differential is evaluated with
noise, showing that if the noise is normally distributed, with sufficient
mass concentrated in a small enough neighborhood of $0$, we can characterize
the expected number of iterations  in order to
attain approximate first order optimality.
\begin{proposition}[Noisy Newton-type Methods]\label{thm:noise}
  Using the same assumptions and notations as in the setting of
  Theorem~\ref{thm:stochastic newton}, let
  $\Sigma \in \mathbb{R}^{n^2 \times n^2}$ be a covariance (positive semidefinite) matrix acting
  on $\bbRnxn$.
  Consider the noisy version of the Newton Differential,
  $\mathcal{G}f(x) = \mathcal{H}f(x) + N$, where
  $N \in \bbRnxn$ is a random variable. Let $K$ the random number of iterations until
  Algorithm~\ref{alg: Stochastic Newton} with ${B^k}^{-1}$ sampled from
  $\mathcal{G}f(x^k)$ terminates.
  If $N$ follows the multivariate normal distribution
  $\mathcal{N}(0, \Sigma)$ and $c + n^{3/2}\sqrt{\tr \Sigma} < \Omega^{-1}$, then
  \begin{equation*}
    \bbE(K) \le \frac{1}{1 - n^{-1}}\left(\frac{f(x^0) - f(\xbar)}{0.5(1-c){(\omega^{-1}\varepsilon)}^2}
      + \frac{\log \frac{0.5(1-c)}{\rho^0}}{\log \alpha } \right).
  \end{equation*}
\end{proposition}
\begin{remark}
  Because the space of non-invertible matrices has Lebesgue measure $0$, the
  probability that $B$ sampled from $\mathcal{G}f(x)$ is not invertible is $0$. This
  allows us to remove the non-invertible matrices from the output of this
  random variable, without changing its distribution.
\end{remark}

For the proof, we first need an auxiliary lemma.
\begin{lemma}\label{lemma:inexact newton}
  If $\mathcal{H}F(x)$ is a Newton differential for $F$, then so is
  $\mathcal{H}f(x) + \bbB_{\beta}[0]$ with constant at most $c + \beta$.
\end{lemma}
\begin{proof}
  This follows simply by using the
  triangle inequality and the definitions, showing that for any $x, y$ and
  $G \in \mathcal{H}f(x) + \bbB_{\beta}[0]$ and $H \in \mathcal{H}F(x)$
  \begin{align*}
    \frac{\|F(x) - F(y) - G(x - y)\|}{\|x - y\|}
      &\le \frac{\|F(x) - F(y) - (G - H + H)(x - y)\|}{\|x - y\|} \nonumber \\
      &\le \frac{\|F(x) - F(y) - H(x - y)\|}{\|x - y\|} \nonumber \\
      &+ \frac{\|(H - G)(x - y)\|}{\|x - y\|} \nonumber \\
      &\le \frac{\|F(x) - F(y) - H(x - y)\|}{\|x - y\|} + \|H - G\| \nonumber \\
      &\le \frac{\|F(x) - F(y) - H(x - y)\|}{\|x - y\|} + \beta.
  \end{align*}
  Because $x, y, H$ and $G$ are arbitrary, taking the supremum and using the
  definition of Newton differentiability proves the desired result.
\end{proof}

We can now prove the main result from this subsection.
\begin{proof}[Proof of Proposition~\ref{thm:noise}]
  From the previous lemma, we know that
  $\mathcal{H}f(x) + \bbB_{n^{3/2}\sqrt{\tr \Sigma}}[0]$ is a Newton differential for $\partial_1 f(x)$ with
  constant at most $c + n^2\sqrt{\tr \Sigma}$.

  Using the properties of the Wishart distribution
  \begin{equation*}
    \bbE(\nu(N){\nu(N)}^T) = n^2 \Sigma,
  \end{equation*}
  and applying the trace, which is a linear operator, we obtain
  \begin{equation*}
    \bbE(\|N\|^2) = n^2 \tr \Sigma.
  \end{equation*}
  Markov's inequality shows that
  \begin{equation*}
    \bbP(\|N\| \ge n^{3/2} \sqrt{\tr \Sigma}) =
    \bbP(\|N\|^2 \ge n^3 \tr \Sigma) \le \frac{n^2 \tr \Sigma}{n^3 \tr \Sigma} = \frac{1}{n}
  \end{equation*}
  holds, so
  \begin{equation*}
    \bbP(\mathcal{G}f(x) \in \mathcal{H}f(x) + \bbB_{n^{3/2}\sqrt{\tr \Sigma}}[0]) \ge 1 - \frac{1}{n}.
  \end{equation*}

  If the noise is small enough, such that $c + n^{3/2}\sqrt{\tr \Sigma} < \Omega^{-1}$,
  we can use Theorem~\ref{thm:stochastic newton}. Recall the formula for
  the constant of sufficient decrease from Lemma~\ref{lemma:sufficient decrease},
  and for the constant of a gradient lower bound from
  Lemma~\ref{lemma: gradient lower bound}. This allows for the explicit
  formula for the expectation of $K$, the random variable representing the
  total number of iterations of Algorithm~\ref{alg: Stochastic Newton}, for
  a given $\varepsilon$. Namely
  \begin{equation*}
    \bbE(K) \le \frac{1}{1 - n^{-1}}\left(\frac{f(x^0) - f(\xbar)}{0.5(1-c){(\omega^{-1}\varepsilon)}^2}
      + \frac{\log \frac{0.5(1-c)}{\rho^0}}{\log \alpha } \right).
  \end{equation*}
\end{proof}

\subsection{X-ray, Free Electron Laser}
This subsection is based on
the mathematical formulation of the X-ray Free Electron Laser (XFEL) of
Luke, Schultze, and Grubm\"uller, in~\cite{LukSchGru24Stoc} and on
the author's quasi-Newton  the numerical algorithms implemented for this
problem.
The Julia code that produced the plots can be found in~\cite{Pin24NewtGit}.

Single-shot femtosecond  X-FEL imaging represents a useful test case for
our stochastic Newton-type methods framework. This is an imaging technique
that aims to recover the electron density of various biomolecules, by
interpreting a diffraction pattern. Mathematically, a molecule can be
represented by a probability distribution on $\bbR^3$. Diffraction, through a
random orientation of the molecule, represented by an element of $\SO(3)$,
induces a probability distribution on the detector plane, represented by $\bbR^2$.
The component-wise amplitude of a Fourier transform represents the likelihood
of detecting a photon at a particular place on the detector plane. A point
process generates around a hundred photon counts using these likelihoods. The
problem we aim to solve is a maximum likelihood estimation, where the
distribution is over photon counts in each discrete pixel of the detector plane,
and the parameter is the representation of the biomolecule. Randomness enters
the system, both through the stochastic nature of photon hits and through
the random orientation of the molecule.
\begin{remark}
  Due to Compton radiation, the pulse of photons that can hit a sample has
  to be very small. As such, we cannot create enough photon samples in order
  to build their probability distribution, via a histogram.
\end{remark}

As far as this work is concerned, the only properties of the problem that
are of importance are the smoothness of the objective and the fact that
oracle information is only estimated randomly. Nonetheless, for
completeness, we present a detailed mathematical formulation. Consider
$\mathcal{E}$ the finite dimensional parameter space and
let $u \in \mathcal{E} \subseteq \mathcal{L}^2(\bbR^3, \bbC)$, where $\SO(3)$ acts on $\mathcal{E}$, and
$\rho = \Unif \SO(3)$. Let $|\mathcal{F}|: \mathcal{E} \to \Pi(\bbR^2)$, where $\Pi(\bbR2)$ is the set of all
probability measures on the plane. In practice this map is related to
the amplitudes of the Fourier transform. This induces a map $\delta:\mathcal{E} \to \Pi(\bbR^2)$
defined by
\begin{equation*}
  \delta_u(A) = \int_{\SO(3)} |\mathcal{F}|(u \circ \rho)(A) \diff \rho.
\end{equation*}
This $\delta_u$ is a parametric probability distribution.
We assume that $\delta_u$ is Lebesgue absolutely continuous for all $u$, and denote
the Radon-Nykodim derivative by
\begin{equation*}
  f(u, \cdot) = \frac{\partial \delta_u}{\partial \lambda}.
\end{equation*}
Consider the maximum likelihood estimator for $k$ photon samples $\mathcal{U}_k:{\bbR^2}^k \to \mathcal{E}$,
defined by
\begin{equation*}
  \mathcal{U}_k(x_1, \dots, x_k) = \argmin_{u \in \mathcal{E}} \sum_{i = 1}^k-\log f(u, x_i).
\end{equation*}
Let $u \in \mathcal{E}$ the true electron distribution and let $X_1, \dots, X_N \sim \delta_u$.
The problem consists in understanding the distribution of
$\mathcal{U}_N(X_1, \dots, X_N)$ as $N$ goes to $\infty$.

In our numerical approach, we consider the objective function $v:\mathcal{E} \to \bbR$ defined
by
\begin{equation*}
  v(u) = \sum_{i = 1}^N-\log f(u, X_i).
\end{equation*}
Consider a fixed batch size $B < N$ and a random sequence of permutations of
$\{1,2,\dots,N\}$, ${\{\sigma^k\}}_{k \in \bbN}$. With this, we define the gradient estimator
\begin{equation*}
  g_k(u) = \nabla\sum_{i = 1}^B-\log f(u, X_{\sigma_k(i)}).
\end{equation*}
Let $\mathcal{S}:\mathcal{E} \times \mathcal{E} \times \bbRnxn \to \bbRnxn$ be a method of producing a quasi-Newton matrix
from an iterate, an estimation of the gradient, and a previous quasi-Newton matrix,
such as Broyden, DFP, or BFGS.
We run an inner loop of quasi-Newton steps with a given dataset
$X_{\sigma(1)},\dots,X_{\sigma(B)}$, followed by resampling the dataset and reinitializing
the quasi-Newton method, but maintaining the current iterate. The points of
interest to us are the points produced by the outer iteration.
Let $M$ be the number of inner
iterations and let $\lfloor k / M \rfloor$ be the integer part of $k/M$, i.e.~the iteration
counter of the outer loop. The sequence of stochastic Newton differentials is
then generated by
\begin{equation*}
  H^{k+1} = \left\{\begin{array}{lr}
    \alpha \Id&\quad \text{ if } M \text{ divides } k, \\
    \mathcal{S} (u^k, g_{\lfloor k / M \rfloor}(u^k), H^k)&\quad \text{ else},
  \end{array} \right.
\end{equation*}
for some algorithm parameter $\alpha$.

In Figures~\ref{fig:SQN a} and~\ref{fig:SQN b}, we see the behavior of this method,
with $M=10$ and $M=300$ respectively, compared
to a more traditional stochastic gradient descent, set up in the same inner and outer
loop way. For constructing the quasi-Newton matrices, we use
Samsara~\cite{WebSamsara,WebSamsarajl},
a reverse communication nonlinear optimization solver for smooth unconstrained
objectives developed by Luke et al. Under standard statistical assumptions,
it is known that
\begin{equation*}
  \lim_{N \to \infty}\frac{1}{N}\sum_{k=0}^{N-1}\mathcal{U}_M(X_{kM+1}, \dots, X_{kM+M}) = \delta_u \aseq,
\end{equation*}
where $u$ is the true representation that has produced the data $X_1,X_2,\dots$.
This is the reason why we plot the means, $k^{-1}\sum_{i=0}^k u^i$. We can
see in the figure, that the variance produced by the stochastic quasi-Newton
method, for $M=300$,
is lower, and this suggests that the inner loop functions as a
significantly better approximation of the maximum likelihood estimator, and thus
the behavior of the outer loop is dominated by statistics and not by
optimization. Paradoxically, the minimal variance is achieved by the
stochastic gradient descent with $M=10$. This happens because the steps
taken by 10 iterations of gradient descent are smaller, producing thus a smaller
variance. When looking at the objective value, the Newton-type method with
$M=300$ is significantly better than that obtained by the other methods.

\begin{figure}[ht!]
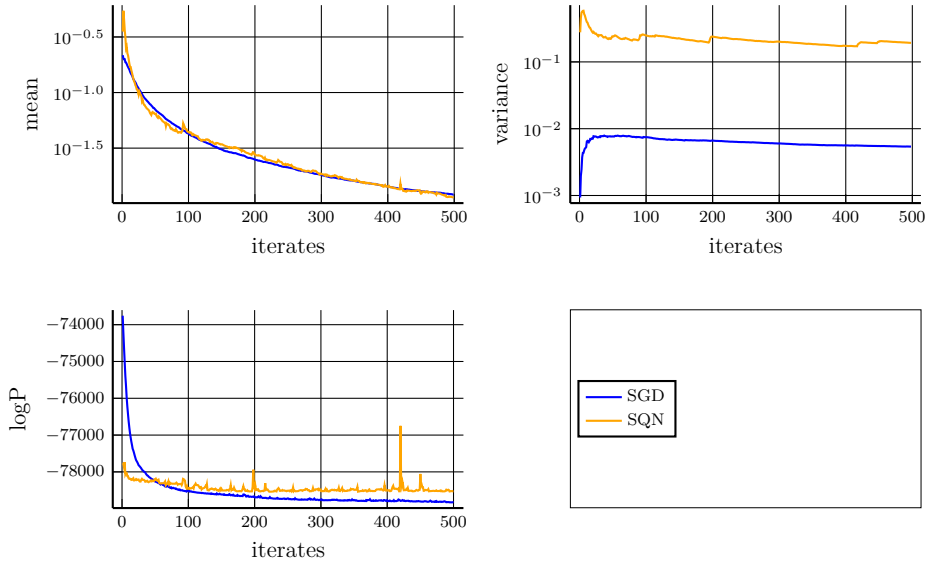

  \begin{adjustbox}{width=\textwidth}%

  \end{adjustbox}
  \caption{Mean and variance over $k$ of the steps $x_{k+1} - x_k$, and the
    objective value compared between Stochastic Gradient Descent and Stochastic
    Quasi-Newton with $M = 10$ inner iterations}%
  \label{fig:SQN a}
\end{figure}

\begin{figure}[ht!]
  \begin{adjustbox}{width=\textwidth}%
%
  \end{adjustbox}
  \caption{Mean and variance over $k$ of the steps $x_{k+1} - x_k$, and the
    objective value compared between Stochastic Gradient Descent and Stochastic
    Quasi-Newton with $M = 100$ inner iterations}%
  \label{fig:SQN b}
\end{figure}

\subsection{Sketched Newton-type Methods}
The Johnson-Lindenstrauss Embedding Lemma has been a tool of choice
in solving large scale optimization systems, see for instance~\cite{CarFowZhe22Rand}.
For completeness, we recall this important result here.
\begin{lemma}[Johnson-Lindenstrauss]\label{lemma:jl}
  For any $n \in \bbN$ and for any $\varepsilon > 0$ and $\delta < 1/2$ and any
  $d \in \mathcal{O}(-\varepsilon^{-2}\log(\delta))$ there exists a distribution over $\bbR^{d \times n}$ such
  that for any vector $x \in \bbRn$ with $\|x\| = 1$
  \begin{equation*}
    \bbP(|\|Ax\|^2 - 1| \le \varepsilon) \le \delta.
  \end{equation*}
\end{lemma}
An example of such a distribution is given by matrices with components
independent and identically distributed from the standard normal distribution.
\begin{remark}
  The dimension $d$ depends on the desired error bound $\varepsilon$. In practice, in our
  algorithmic implementation, the dimension be given and the error analysis
  follow from this fixed dimension.
\end{remark}

This approach involves modifying an iterative solver by using a randomly
selected sketching matrix $\bbR^{d \times n}$ to reduce the dimension of the
objects involved. The computational advantages stem from the fact that $d \ll n$,
so all the numerical linear algebra subroutines should be significantly faster.

In our Newton-type methods, we employ sketching both for the function,
i.e.~$F(x) \mapsto S F(x)$ and for the Newton differential, i.e.~$H \mapsto S^T HS$.
Using this approach yields a step in $\bbR^d$, and in order to map this
step back to $\bbRn$ we use the $S^T$ map. Next, we put this intuitive
explanation into a formal definition.
\begin{definition}[Sketched Newton-type Operator]
  Let $F:U \subseteq \bbRn \to \bbRn$ be pointwise weakly Newton differentiable at $\xbar$ with
  Newton differential $\mathcal{H}F$, and let $d \ll n$.
  The fixed point iteration of the set-valued operator
  $\mathcal{N}_{\mathcal{H}F}:\bbRn \times \bbR^{d \times n}\setto \bbRn$, defined by
  \begin{equation*}
    \mathcal{N}_{\mathcal{H}F} (x, S) = \{x - S^T{(S HS^T)}^{-1} SF(x)~|~H \in \mathcal{H}F(x),\det S H S^T \ne 0\},
  \end{equation*}
  is called {\em a sketched Newton-type method}.
\end{definition}

\begin{lemma}
  Let $F:U \subseteq \bbRn \to \bbRn$ be pointwise weakly Newton differentiable at $\xbar$ with
  $F(\xbar) = 0$ and
  Newton differential $\mathcal{H}F$, and let $d \ll n$. Let $\mathcal{S}$ be a distribution over
  $\bbR^{d \times n}$ as in Lemma~\ref{lemma:jl} and $c > 0$ be such that
  \begin{equation*}
    \forall x \in U,\quad \sup_{H \in \mathcal{H}F(x)}\|F(x) - H(x - \xbar)\| \le c\|x - \xbar\|.
  \end{equation*}
  Let $y \in U$ and $H \in \mathcal{H}F(y)$ and let $S$ be sampled from $\mathcal{S}$, and assume further
  that there exist $\varepsilon > 0$ and $\delta > 0$ such that
  \begin{equation}\label{eq:PCA}
    \bbP\left(\begin{array}{lr}
      \|S^T SHS^T S - H\| \le \varepsilon &\text{ and} \\
      \|y - \xbar - S^T{(S^T HS)}^{-1} SF(y)\| \le (1 + \varepsilon)
      \|Sy - S\xbar - {(S^T HS)}^{-1} SF(y)\|
                            &\text{ and} \\
      \|SF(y) - (S H S^T) S(y - \xbar)\| \le (1 + \varepsilon)\|F(y) - S^T (S H S^T) S(y - \xbar)\|
                            &
    \end{array}\right) \ge 1 - \delta.
  \end{equation}
  Let $y^{+} \in \mathcal{N}_{\mathcal{H}F}(y, S)$ be defined by
  \begin{equation}\label{eq:sketching yp}
    y^{+}=y - S^T{(SHS^T)}^{-1} SF(y).
  \end{equation}
  Then
  \begin{equation}\label{eq:sketching result}
    \bbP(\|y^{+} - \xbar\| \le {(1 + \varepsilon)}^2(c + \varepsilon)\|{(SHS^T)}^{-1}\|\|y - \xbar\|) \ge 1 - \delta.
  \end{equation}
\end{lemma}
\begin{proof}
  Let $E$ denote the event inside the probability in~\eqref{eq:PCA}. If
  $E$ implies the event inside the probability in~\eqref{eq:sketching result},
  then we can deduce
  \begin{equation}\label{eq:scketcing prob interpretation}
    \bbP(E) \le \bbP(\|y^{+} - \xbar\| \le {(1 + \varepsilon)}^2(c + \varepsilon)\|{(SHS^T)}^{-1}\|\|y - \xbar\|).
  \end{equation}
  Indeed, we now show that $E$ implies the event inside the
  probability in~\eqref{eq:sketching result}. To show this implication,
  we assume that $S$ is such that
  \begin{equation}\label{eq:sketching cool pca}
    \|S^T SHS^T S - H\| \le \varepsilon,
  \end{equation}
  and
  \begin{equation}\label{eq:sketching 1}
        \|SF(y) - (S H S^T) S(y - \xbar)\|
        \le (1 + \varepsilon)\|F(y) - S^T (S H S^T) S(y - \xbar)\|,
  \end{equation}
  and
  \begin{equation}\label{eq:sketching 2}
    \|y - \xbar - S^T{(SHS^T)}^{-1} SF(y)\| \le (1 + \varepsilon)
    \|Sy - S\xbar - {(SHS^T)}^{-1} SF(y)\|.
  \end{equation}
  Using~\eqref{eq:sketching 2}, we can express
  \begin{equation*}
    \|y - \xbar - S^T{(SHS^T)}^{-1} SF(y)\| \le (1 + \varepsilon)
    \|{(SHS^T)}^{-1}\|\|{(SHS^T)}S(y - \xbar) - SF(y)\|,
  \end{equation*}
  and recalling the definition of $y^{+}$ from~\eqref{eq:sketching yp},
  \begin{equation}\label{eq:sketching 3}
    \|y^{+} - \xbar\| \le (1 + \varepsilon)
    \|{(SHS^T)}^{-1}\|\|{(SHS^T)}S(y - \xbar) - SF(y)\|.
  \end{equation}
  Combining~\eqref{eq:sketching 3} with~\eqref{eq:sketching 1} gives
  \begin{equation}\label{eq:sketching also close}
    \|y^{+} - \xbar\| \le {(1 + \varepsilon)}^2\|{(SHS^T)}^{-1}\|
    \|F(y) - F(\xbar) - S^T (S H S^T) S(y - \xbar)\|.
  \end{equation}
  Using~\eqref{eq:sketching cool pca} shows that
  \begin{equation*}
    S^T (S H S^T) \in B_{\varepsilon}[H],
  \end{equation*}
  and Lemma~\ref{lemma:inexact newton} gives
  \begin{equation}\label{eq:sketching soooo close}
    \|F(y) - F(\xbar) - S^T (S H S^T) S(y - \xbar)\| \le (c + \varepsilon) \|y - \xbar\|.
  \end{equation}
  Finally, using~\eqref{eq:sketching soooo close}
  in~\eqref{eq:sketching also close} yields
  \begin{equation*}
    \|y^{+} - \xbar\| \le {(1 + \varepsilon)}^2(c + \varepsilon)\|{(SHS^T)}^{-1}\|\|y - \xbar\|.
  \end{equation*}
  Returning to the probability interpretation
  from~\eqref{eq:scketcing prob interpretation}, we complete the proof,
  computing
  \begin{equation*}
    1 - \delta \le \bbP(E) \le \bbP(\|y^{+} - \xbar\| \le {(1 + \varepsilon)}^2(c + \varepsilon)\|{(SHS^T)}^{-1}\|\|y - \xbar\|).
  \end{equation*}
\end{proof}

\begin{remark}
  In~\cite{CunGha14Line}, it was shown that the $d$-principal component analysis
  (PCA) of a matrix $H$ can be computed by
  \begin{equation*}
   \minprob{S \in \bbR^{d \times n}}{\|H - S^T S H\|.}
  \end{equation*}
  This suggests that sketched Newton-type methods behave best, when the sketching
  mapping behaves like the PCA of the Newton differential.
\end{remark}

We apply the sketched Newton-type method to the problem of denoising tubulin
images coming from~\cite{BenAlaTamShrRozProMur17Cont,LukChaSheMal}. This is a
multicriteria
optimization problem, with the state space consisting of $m_1 \times m_2$ real matrices.
We identify this space with $\bbR^{m_1 m_2}$. Let $o \in \bbR^{m_1 \times m_2}$ be the original
image.
The two objectives are given by the functions $n: \bbR^{m_1 \times m_2} \to \bbR$ and
$d:\bbR^{m_1 \times m_2} \to \bbR$.
The first function measures the noise level of the image and is defined as
\begin{equation*}
  n(x) = \|\nabla x\|^2 := \sum_{i = 2}^{m_1 - 1}\sum_{i = 2}^{m_1-1} {(x_{i,j+1} - x_{i,j-1})}^2
  + {(x_{i+1,j} - x_{i-1,j})}^2,
\end{equation*}
and the second one measures the distance to the original image,
$d(x) = \|x - o\|_F^2$.
To associate to it a scalar optimization problem that can be solved using the
approach proposed in this paper, we employ a
scalarization strategy, yielding
\begin{equation*}
  \min_{x \in \bbR^{m_1 \times m_2}}{n(x) + \alpha d(x)}
\end{equation*}
for some parameter $\alpha \in (0, \infty)$. How such a parameter $\alpha$ is chosen belongs to
multicriteria optimization, and is beyond the scope of this work. For the
interested reader, the author recommends~\cite{KhaZal15Setv,BotGraWan09Dual} for
a reference on how to choose such a parameter.
Both functions are smooth, so we can use the gradient and the Hessian in
implementation of the algorithm. Our framework is then required in order to handle
the randomness steaming from the sketching approximation.

In Figure~\ref{fig:sketched newton} we see the convergence of the algorithm when
$d$, the dimension from the Johnson-Lindenstrauss Embedding Lemma~\ref{lemma:jl},
varies through different proportions of the full dimension $m_1m_2$. We also
see a Newton-type method without sketching for comparison. The
sketching distribution, $\mathcal{S}$, from which $S$ is drawn is $\mathcal{N}(0, \Id)$
in Subfigure~\ref{fig:sketched a} and $\Unif \Pi(n,d)$
in Subfigure~\ref{fig:sketched b}. Here the set $\Pi(n,d)$ represents the canonical
projection matrices, i.e.
\begin{equation*}
  \Pi(n,d) = \left\{(p_{ij}) ~\left|~ \forall i \le d,j \le n,\quad p_{ij} \in \{0, 1\},\quad \sum_{j=0}^n p_{ij} = 1 \right.\right\}.
\end{equation*}

\begin{figure}[ht!]
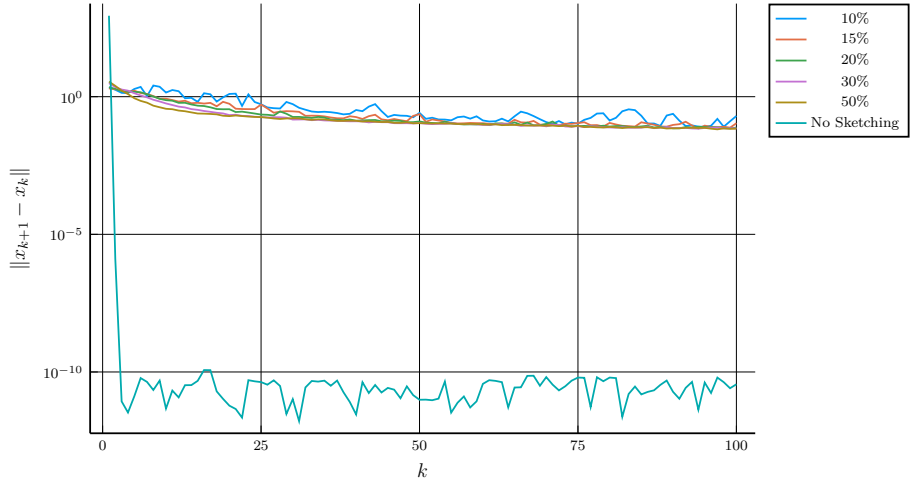
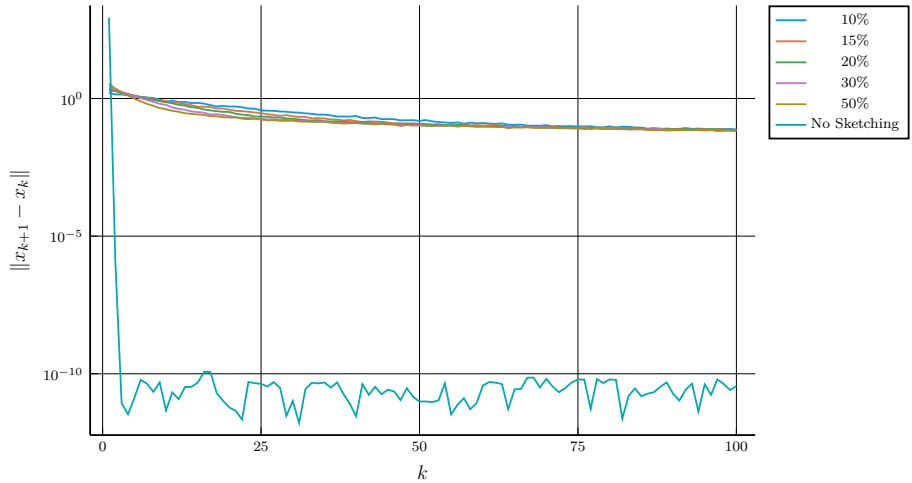

  \begin{subfigure}{0.97\textwidth}
    \begin{adjustbox}{width=\textwidth}%

    \end{adjustbox}
    \caption{The probability distribution of the embedding matrices is $\mathcal{N}(0, \Id)$}%
    \label{fig:sketched a}
  \end{subfigure}
  \begin{subfigure}{0.97\textwidth}
    \begin{adjustbox}{width=\textwidth}%
%
    \end{adjustbox}
    \caption{The probability distribution of the embedding matrices is $\Unif \Pi(n,d)$}%
    \label{fig:sketched b}
  \end{subfigure}
  \caption{The step size of the sketched Newton algorithm with the dimension
    of the embedding space given as a percentage of the full space compared to
    a standard Newton method}%
  \label{fig:sketched newton}
\end{figure}
\begin{remark}
  Using a sketching matrix from $\Pi(n,d)$ corresponds to selecting $d$ elements
  from the gradient and $d^2$ from the Hessian.
\end{remark}

\section{Conclusions}
Contrary to the standard view in the machine learning community, \lparan{}Quasi-\rparan{}Newton
methods can be successfully employed in big data environments. We have seen
that the fast convergence of these methods can remove uncertain behavior
due to inner optimization loops, thus the dynamics of the system can be interpreted
from a purely statistical perspective.
This paper shows that a backtracking approach can be used to maintain strong
guarantees on the behavior of Newton-type methods. As a future research
direction different backtracking strategies can be investigated.
Finally, our approach only handles randomness in the Hessian approximation.
While some structured noise in the objective can be seen as noise in the
Hessian approximation, not all random estimators of the objective can
easily be treated by our methods. As such, an investigation of a fully
stochastic method, with a stochastic objective along the lines presented here
would be of great interest.

\section{Acknowledgments}
This work is partially the PhD thesis of the author, completed at
the University of G\"ottingen under the supervision of D. Russell Luke. The author
would like to thank D. Russell Luke for his guidance and advice during the
undertaking of the PhD. The author would also like to thank his
postdoc mentor, Sorin-Mihai Grad. This work has been partially founded
by ANR-22-EXES-0013.

\end{document}